\documentclass[12pt, reqno]{amsart}
\usepackage{amsmath, amsthm, amscd, amsfonts, amssymb, graphicx, graphics,color, mathrsfs, epsfig, latexsym}
\usepackage[plainpages,breaklinks,colorlinks=true,allcolors=blue,citecolor=red,pagebackref=false,hyperindex=false]{hyperref}
\definecolor{pink}{rgb}{1,0.03,0.8}
\usepackage[margin=2.9cm]{geometry}
\usepackage{t1enc}
\usepackage[noadjust]{cite}

\newtheorem{thm}{Theorem}[section]
\newtheorem{cor}[thm]{Corollary}
\newtheorem{lem}[thm]{Lemma}

\theoremstyle{definition}

\theoremstyle{remark}

\numberwithin{equation}{section}

\newcommand{\V}{\mathscr{V}}

\newcommand{\p}{\mathscr{P}}

\title [Stability of the Pexiderized Cauchy Functional Equation]{Hyers-Ulam Type Stability of the Pexiderized Cauchy Functional Equation in Locally Convex Cones}
\author [J.  Mohammadpour,  A. Najati and Iz. EL-Fassi]{Jafar Mohammadpour,  Abbas Najati and Iz-iddine EL-Fassi}
\date{}
\address{Jafar Mohammadpour, Abbas Najati
\newline
 \indent
 Department of Mathematics, Faculty of Mathematical Sciences
\newline
\indent  University of  Mohaghegh Ardabili,  Ardabil, Iran}
\email{jm8427836@gmail.com; a.nejati@yahoo.com}
\date{}
\address{Iz-iddine EL-Fassi \newline
 \indent
Department of Mathematics, Faculty of Sciences and Techniques,
\newline
\indent S. M. Ben Abdellah University, Fez, Morocco}
\email{ izidd-math@hotmail.fr,~ izelfassi.math@gmail.com,
\newline
\indent\hspace{2.5cm} iziddine.elfassi@usmba.ac.ma}

\begin{document}

\maketitle

\begin{abstract}
The foundation of locally convex cone theory relies on order-theoretic concepts that induce specific topological frameworks. Within this structure, cones naturally possess three distinct topologies: lower, upper, and symmetric. In this paper, we consider the Hyers-Ulam type stability of the Pexiderized Cauchy functional equation $f(x+y)=g(x)+h(y)$ in locally convex cones. Additionally, we present several significant corollaries that follow from our primary findings.
\end{abstract}

\section{Introduction and preliminaries}

The concept of locally convex cones was first introduced and further developed in \cite{KR} and \cite{Ro}. A cone, denoted by
 $\p$, is a set equipped with two operations: addition $(a,b)\mapsto a+b$ and scalar multiplication
 $(\lambda,a)\mapsto \lambda a$, where   $\lambda$ is a non-negative real number. The addition operation is assumed to be both associative and commutative, and there exists a neutral element $0\in\p$.
For scalar multiplication, the standard associative and distributive properties are satisfied. Specifically, for all $a,b\in\p$ and $\lambda,\mu\geqslant0$ the following hold:
\[
\lambda(\mu a) = (\lambda\mu)a, \quad (\lambda + \mu)a = \lambda a + \mu a, \quad \lambda(a + b) = \lambda a + \lambda b, \quad 1a = a, \quad \text{and} \quad 0a = 0
\]
It is important to note that the cancellation law, which states that
 $a+c = b+c$ implies $a = b$, is not generally required in this framework.
 \par
A \textit{preordered cone} is a cone \(\p \) equipped with a reflexive and transitive relation $\leqslant$ which is compatible with both the addition and scalar multiplication operations. Specifically, for all $a,b,c \in \p $ and $\lambda\geqslant0$, the following properties hold:
\[
a\leqslant b \implies a+c \leqslant b+c \quad \text{and} \quad \lambda a \leqslant \lambda b.
\]
It is worth noting that the relation $\leqslant$ is not required to satisfy anti-symmetry, meaning that $a\leqslant b$ and $b\leqslant a$ do not necessarily imply
$a=b$.
Every ordered vector space naturally qualifies as a preordered cone. Additionally, the cones $\mathbb{\overline{R}} = \mathbb{R} \cup \{+\infty\}$ and $\mathbb{\overline{R}}_+ = [0,+\infty]$, when equipped with the standard order and algebraic operations (notably, $0 \cdot (+\infty) = 0)$ are examples of preordered cones. Every cone $\p$ can be endowed with a natural preorder, defined by the relation
$a\leqslant b$ if there exists an element $c\in\p$ such that $a+c=b$. This preorder structure is inherent to the cone and does not require additional assumptions.
\par
A subset $\V$ of a preordered cone $\p$ is referred to as an \textit{abstract 0-neighborhood system} if it satisfies the following conditions:
\begin{itemize}
  \item [$\rm(i)$] $0 < v$ \text{ for all } $v \in \V$;
  \item [$\rm(ii)$] For any $u, v \in \V$,  there exists an element  $w \in \V$ such that
  $w \leqslant u$ \text{ and } $w\leqslant v$;
  \item [$\rm(iii)$] $u + v \in \V$ \text{ and } $\lambda v \in \V$ \text{ whenever } $u, v \in \V$ \text{ and } $\lambda > 0$.
\end{itemize}
The elements
$v$ of $\V$ define upper and lower neighborhoods for the elements of $\p$ as follows:
\begin{align*}
  v(a) &= \{b \in \p:~ b \leqslant a + v\}\quad (\text{\rm upper neighborhood}), \\
 (a)v &= \{b \in \p:~   a \leqslant b + v\}\quad (\text{\rm lower neighborhood}).
\end{align*}
The intersection $v(a)v:= v(a)\cap(a)v$ defines the symmetric neighborhood of $a$. These neighborhoods generate the upper, lower, and symmetric topologies on $\p$, respectively.
\par
Let $\p$ be a preordered cone, and let $\V\subseteq\p$ be an abstract 0-neighborhood system. The pair $(\p,\V)$ is referred to as a \textit{full locally convex cone}. Due to the inherent asymmetry of cones, asymmetric conditions naturally arise. For technical reasons, we require that the elements of a full locally convex cone  $(\p,\V)$ are bounded below.  This means that for every $a \in \p $ and $v \in \V$, there exists $\rho > 0$ such that $0 \leqslant a + \rho v$.
Any subcone of $(\p,\V)$, even if it does not include the abstract $0$-neighborhood system $\V$, is termed a \textit{locally convex cone}.
 An element $a \in \p$ is called \textit{upper bounded} if, for every $v \in \V$,  there exists a positive scalar $\lambda > 0$ such that $a \leqslant \lambda v$. The element $a\in\p$ is said to be \textit{bounded} if it is both lower and upper bounded. By defining $\xi = \{\varepsilon > 0 :~\varepsilon \in \Bbb{R}\}$, the pairs $(\Bbb{\overline{R}}, \xi)$ and $(\Bbb{\overline{R}}_+, \xi)$ form examples of full locally convex cones.
 \par
 Let $(\p , \V)$ be a locally convex cone, and let  $a \in \p$. The closure of $a$ is defined as the intersection of all its upper neighborhoods, that is,
\[
\overline{a}: = \cap\{v(a):~ v \in \V\}.
\]
It is straightforward to verify that $\bar{a}$ represents the closure of the singleton set $\{a\}$ under the lower topology \cite[Corollary I.3.5]{KR}.
A locally convex cone  $(\p , \V)$ is said to be \textit{separated}  if, for any $a,b\in\p$, the equality $\overline{a}=\overline{b}$
  implies $a=b$. In other words, distinct elements in $\p$ must have distinct closures. The locally convex cone $(\p , \V)$ is separated if
and only if the symmetric topology on $\p$ is Hausdorff (see \cite[Proposition I.3.9]{KR}).
\par
For a net $(a_i)_{i \in I}$ in $(\p,\V)$, we say that it converges to an element $a \in \p$  with respect to the symmetric relative topology of $\p$ if,
 for every $v\in\V$, there exists an index $i_0 \in I$ such that $a_i\leqslant a+ v$ for all $i \geqslant i_0$. The net $(a_i)_{i \in I}$ is called a (symmetric) \textit{Cauchy net} if for every $v \in \V$, there exists an index $i_0 \in I$ such that $a_i \leqslant a_j+ v$ for all $i,j \geqslant i_0$. Clearly, if a net converges, it must also be a Cauchy net. A locally convex cone $(\mathscr{P}, \mathcal{V})$ is said to be (symmetric)  \textit{complete} if every (symmetric) Cauchy net in $(\p,\V)$ converges to an element in $\p$.
\par
A locally convex cone $(\p, \V)$ is called a \textit{uc-cone} if $\V = \{\lambda w : \lambda > 0\}$ for some $w \in \V$. In this case, the element $w$ is referred to as the \textit{generating element} of $\V$, and all elements of $\V$ are bounded . If $(\p, \V)$ is a uc-cone and $\p$ is simultaneously a real vector space,
then $\p$ becomes a seminormed space under the symmetric topology of $(\p, \V)$. When $\V = \{\lambda w : \lambda > 0\}$, the seminorm $q: \p \to [0, +\infty]$ is defined as:
\begin{align}\label{norm}
q(a) = \inf\{\mu > 0 : \mu^{-1} a \in w(0)w\}, \quad a \in \mathscr{P}.
\end{align}
If the symmetric topology on $\p$ is Hausdorff, then $q$ becomes a norm on $\p$ (see \cite{AR} for further details).
\par
Let us recall that an equation is said to be Ulam stable in a class of mappings provided each mapping from this class fulfilling our equation "approximately" is "near" to its actual solution. It is well known that the problem of the stability of homomorphisms of metric groups (in other words, the Cauchy functional equation) was posed by  Ulam \cite{Ulam1940} in 1940. In the next year Hyers \cite{Hyers1941} gave the
first affirmative answer of the Ulam's problem for Banach spaces.
\par
Hyers's Theorem was generalized by Aoki \cite{Aoki1950} for additive mappings in 1950, and independently, by Rassias \cite{Rassias1978} in 1978 for linear mappings considering the Cauchy difference controlled by the sum of powers of norms. This type of stability is called \textit{Hyers-Ulam-Rassias stability}. On the other hand, Rassias \cite{Rassias1982} considered the Cauchy difference controlled by the product of different powers of norms.  G\v{a}vruta \cite{Gavruta1994} has generalized the result of  Rassias \cite{Rassias1978}, which permitted the Cauchy difference to become arbitrary unbounded. The stability problems of several functional equations have been extensively investigated in various spaces by a number of authors, and a large list of references can be found, for example, in \cite{BK, El-Fassi2023, FZ, forti, Jung1998,NR1,NR2, Nik}.
\par
In the recent decades, the stability of functional equations have been investigated by many mathematicians. They have so many applications in Information Theory, Physics, Economic Theory, Pure Mathematics, and Social and
Behavior Sciences. These applications are sources of inspiration, attraction, and
influences to bring more and more people to this interesting and ever-growing
branch of mathematics (see, e.g. \cite{ander,dera}).
\par
This research aims to study the stability in the Ulam sense  of the Pexiderized Cauchy functional equation $f(x+y)=g(x)+h(y)$ in locally convex cones.  As  consequences,  we give some important  corollaries from our main result.
\section{Main Results}
We begin this section with the following  lemma.
\begin{lem}\label{Lem.ab}
Let $ (\p , \V) $ be a separated locally convex cone, and let $a,b\in\p$. If $a\leqslant b+v$ and $b\leqslant a+v$ hold for every $v\in\V$, then $a=b$.
\end{lem}
\begin{proof}
  To prove $a=b$, it suffices to show that $\overline{a}=\overline{b}$. Let $x\in\overline{a}$. Then $x\leqslant a+\frac{1}{2}v$ for all $v\in\V$. Since $a\leqslant b+\frac{1}{2}v$, it follows that $x\leqslant b+v$ for all $v\in\V$,  which implies $x\in\overline{b}$. This demonstrates that $\overline{a}\subseteq\overline{b}$. By a similar argument, we can show that $\overline{a}\subseteq\overline{b}$, leading to $\overline{a}=\overline{b}$. Because $ (\p , \V) $ is a separated locally convex cone, we conclude that $a=b$.
\end{proof}
\begin{cor}
Let $ (\p , \V) $ be a separated locally convex cone, and let $a,b\in\p$. If $a+v= b+v$  for every $v\in\V$, then $a=b$.
\end{cor}
The following lemma has an essential role in our main results.
\begin{lem}{\rm(\cite[Lemma 1]{NR1})}\label{Lem.con}
Let $ (\p , \V) $ be a locally convex cone, and let $ a \in \p  $. Suppose $ \{\lambda_n\} $ is a sequence of non-negative scalars converging to $0$ as $ n \to \infty $. Then $a$ is bounded if and only if $ \lambda_n a $ converges to $0$ as $ n \to \infty $ in the symmetric topology.
\end{lem}
In \cite{NR1}, the stability problem in locally convex cones, in the sense of Hyers-Ulam, is investigated, and the stability of linear operators in locally convex cones is discussed. The following theorem presents a slightly modified version of Theorem 1 in \cite{NR1}.
\begin{thm}{\rm(\cite[Theorem 1]{NR1})}
  Let $ (\p _1, \V_1) $ be a locally convex cone, and let $ (\p _2, \V_2) $ be a separated full locally
convex cone.  Assume that $ (\p _2, \V_2) $ is complete  under the symmetric topology and that the preorder in  $\p_2$ is antisymmetric. If a function $f:\p_1\to\p_2$ satisfies
\[f(\lambda x + y)\in v(\lambda f(x) + f(y))v\]
for all $x,y\in\p_1$ and all $\lambda>0$, where $v\in\V$ is a fixed bounded element, then there exists a unique linear operator $\ell:\p_1\to\p_2$
such that $\ell(x)\in v(f(x))v$ for all $x\in\p_1$.
\end{thm}
Furthermore, the Hyers-Ulam type stability of the Jensen operator in locally convex cones was established in \cite{NR2} under specific assumptions.
\begin{thm}{\rm(\cite[Theorem 1]{NR2})}
Consider $ (\p _1, \V_1) $ as a locally convex cone and $ (\p _2, \V_2) $ as a separated, full locally convex cone. Assume that $ (\p _2, \V_2) $ is complete with respect to the symmetric topology and that the preorder defined on  $ \p _2 $ is antisymmetric. If a mapping $ f : \p _1 \to \p _2 $ satisfies the condition
\[
2 f\left(\frac{x + y}{2}\right) \in v(f(x) + f(y))v
\]
for some bounded element $ v \in \V_2 $ and for all $x, y \in \p _1$, and if  $ f(0) $ is bounded, then there exists a unique additive Jensen operator $ g : \p _1 \to \p _2 $ and a positive real number $ \gamma $ such that
\[
g(x) \in (\gamma v)(f(x))(\gamma v)
\]
for all $ x \in \p _1 $.
\end{thm}
We now provide our main result.
\begin{thm}\label{Px}
Consider $ (\p _1, \V_1) $ as a locally convex cone and $ (\p _2, \V_2) $ as a separated, full locally convex cone and complete with respect to the symmetric topology. Assume that  the mappings $ f,g,h : \p _1 \to \p _2 $ satisfy the condition
\begin{equation}\label{Px.1}
f(x+y) \in v(g(x) + h(y))v
\end{equation}
for some bounded element $ v \in \V_2 $ and for all $x, y \in \p _1$, and   $ f(0) $ is bounded. Then there exists a unique additive mapping $A:\p _1 \to \p _2$ and a positive real number $ \delta $ such that
\begin{align*}
 A(x)&\in (\delta v)(f(x))(\delta v), \\
 A(x)&\in (1+\delta)v(g(x)+h(0))(1+\delta)v,\\
 A(x)&\in (1+\delta)v(h(x)+g(0))(1+\delta)v
\end{align*}
for all $ x \in \p _1 $.
\end{thm}
\begin{proof}
By substituting $x=0$ and $y=0$ in \eqref{Px.1} separately into \eqref{Px.1}, we obtain
\begin{align}
 f(x)&\leqslant v+g(x)+h(0),\quad g(x)+h(0)\leqslant v+f(x),\quad x\in\p_1\label{p.Px.1}\\
  f(y)&\leqslant v+g(0)+h(y),\quad g(0)+h(y)\leqslant v+f(y),\quad y\in\p_1.\label{p.Px.2}
\end{align}
Therefore, $f(0)\leqslant v+g(0)+h(0)$, and from \eqref{p.Px.1} and \eqref{p.Px.2}, it follows that
\begin{equation}\label{p.Px.3}
\begin{aligned}
  f(0)+f(x+y)&\leqslant v+g(0)+h(0)+f(x+y) \\
   & \leqslant 2v+g(0)+h(0)+g(x) +h(y)\\
   & \leqslant 4v+f(x)+f(y)
\end{aligned}
\end{equation}
for all $x,y\in\p_1$. On the other hand, $g(0)+h(0)\leqslant v+f(0)$, and by applying \eqref{p.Px.1} and \eqref{p.Px.2}, we deduce that
\begin{equation}\label{p.Px.4}
\begin{aligned}
  f(x)+f(y)&\leqslant 2v+g(x)+h(y)+g(0)+h(0) \\
   & \leqslant 3v+f(x+y)+g(0)+h(0)\\
   & \leqslant 4v+f(x+y)+f(0)
\end{aligned}
\end{equation}
for all $x,y\in\p_1$. By defining $w=4v$, the inequalities \eqref{p.Px.3} and \eqref{p.Px.4} lead to
\begin{equation}\label{p.Px.5}
  f(0)+f(x+y)\in w(f(x)+f(y))w,\quad x,y\in\p_1.
\end{equation}
By setting $y=x$  in \eqref{p.Px.5}, we obtain
\begin{equation}\label{p.Px.6}
  f(0)+f(2x)\leqslant w+2f(x),\quad 2f(x)\leqslant w+f(0)+f(2x),\quad x\in\p_1.
\end{equation}
Since $f (0)$is bounded, there exists a positive real number $\lambda>0$ such that
\[f(0)\leqslant \lambda w,\quad 0\leqslant f(0)+\lambda w.\]
Consequently, the inequalities in \eqref{p.Px.6} lead to
\begin{equation}\label{p.Px.6+}
  f(2x)\leqslant (\lambda+1)w+2f(x),\quad 2f(x)\leqslant (\lambda+1)w+f(2x),\quad x\in\p_1.
\end{equation}
We use mathematical induction on $n$ to establish the following inequalities for all $x\in\p_1$ and all $n\in\Bbb{N}$:
\begin{equation}\label{p.Px.7}
  \frac{f(2^nx)}{2^n}  \leqslant f(x)+\left(1-\frac{1}{2^n}\right)(\lambda+1)w
\end{equation}
and
\begin{equation}\label{p.Px.8}
f(x)  \leqslant \frac{f(2^nx)}{2^n}+\left(1-\frac{1}{2^n}\right)(\lambda+1)w.
\end{equation}
 We begin by proving \eqref{p.Px.7}. For the case $n=1$, the inequality \eqref{p.Px.7} follows directly from the first relation in \eqref{p.Px.6}. Now, assume that \eqref{p.Px.7} holds for some  $n$ and all $x\in\p_1$. Then for $n+1$, we have
\begin{align*}
   \frac{f(2^{n+1}x)}{2^{n+1}}&\leqslant \frac{f(2x)}{2}+\left(\frac{1}{2}-\frac{1}{2^{n+1}}\right)(\lambda+1)w\quad\text(\rm by~\eqref{p.Px.7})\\
  &\leqslant \frac{1}{2}(\lambda+1)w+f(x)+\left(\frac{1}{2}-\frac{1}{2^{n+1}}\right)(\lambda+1)w\quad\text(\rm by~\eqref{p.Px.6+})\\
  &=f(x)+\left(1-\frac{1}{2^{n+1}}\right)(\lambda+1)w,\quad x\in\p_1.
\end{align*}
This completes the proof of \eqref{p.Px.7}. We now turn to proving \eqref{p.Px.8}. For the case $n=1$, the inequality \eqref{p.Px.8} is a direct consequence of the second relation in \eqref{p.Px.6}. Next, assume that \eqref{p.Px.8} holds for some  $n$ and all $x\in\p_1$. Then for $n+1$, we derive:
\begin{align*}
   f(x)&\leqslant \frac{1}{2}(\lambda+1)w+\frac{f(2x)}{2}\quad\text(\rm by~\eqref{p.Px.6+})\\
   &\leqslant\frac{1}{2}(\lambda+1)w+\frac{f(2^{n+1}x)}{2^{n+1}}+\left(\frac{1}{2}-\frac{1}{2^{n+1}}\right)(\lambda+1)w \quad\text(\rm by~\eqref{p.Px.8})\\
  &= \frac{f(2^{n+1}x)}{2^{n+1}}+\left(1-\frac{1}{2^{n+1}}\right)(\lambda+1)w,\quad x\in\p_1.
\end{align*}
This completes the proof of \eqref{p.Px.8}. By substituting $x$ by $2^mx$ in \eqref{p.Px.7} and \eqref{p.Px.8}, and then multiplying both sides of the resulting inequalities by $\frac{1}{2^m}$, we obtain
\begin{equation}\label{p.Px.9}
  \frac{f(2^{n+m}x)}{2^{n+m}} \leqslant \frac{f(2^mx)}{2^m}+\frac{1}{2^m}(\lambda+1)w
\end{equation}
and
\begin{equation}\label{p.Px.10}
\frac{f(2^mx)}{2^m} \leqslant \frac{f(2^{n+m}x)}{2^{n+m}}+\frac{1}{2^m}(\lambda+1)w
\end{equation}
for all $x\in\p_1$ and all $m,n\in\Bbb{N}$.
Let $u\in\V$ be an arbitrary element. Since $(\lambda+1)w$ is bounded, Lemma \ref{Lem.con} implies that the sequence $\{\frac{1}{2^m}(\lambda+1)w\}_m$ converges to $0$ as $ m \to \infty $ in the symmetric topology. Consequently, there exists an integer $N>0$ such that $\frac{1}{2^m}(\lambda+1)w\leqslant u$ for all $m>N$. Thus, \eqref{p.Px.9} and \eqref{p.Px.10} lead to
\[
   \frac{f(2^{n+m}x)}{2^{n+m}} \leqslant \frac{f(2^mx)}{2^m}+u,\quad
   \frac{f(2^mx)}{2^m} \leqslant \frac{f(2^{n+m}x)}{2^{n+m}}+u
\]
for all $x\in\p_1$ and all $m,n\in\Bbb{N}$ with $m>N$. Thus, the sequence $\left\{\frac{f(2^nx)}{2^n}\right\}_n$ is Cauchy in $(\p_2,\V_2)$ under the symmetric topology.
Given that $(\p_2,\V_2)$ is complete, this sequence converges with respect to the symmetric topology. Moreover, since $(\p_2,\V_2)$ is separated,  the symmetric topology is Hausdorff, ensuring that the limit of the sequence is unique. Define the mapping $A:\p_1\to \p_2$ as
\[A(x):=\lim_n \frac{f(2^nx)}{2^n},\quad x\in\p_1.\]
From \eqref{p.Px.7} and \eqref{p.Px.8}, we derive the following inequalities:
\begin{equation}\label{p.Px.11}
  \frac{f(2^nx)}{2^n}  \leqslant f(x)+(\lambda+1)w,\quad f(x)  \leqslant \frac{f(2^nx)}{2^n}+(\lambda+1)w
\end{equation}
which hold for all $x\in\p_1$ and all $n\in\Bbb{N}$. As the sequence $\left\{\frac{f(2^nx)}{2^n}\right\}_n$ converges to $A(x)$  in the symmetric topology as $ n \to \infty $, there exists an integer $N>0$ such that
\[\frac{f(2^nx)}{2^n}\leqslant A(x)+w,\quad A(x)\leqslant \frac{f(2^nx)}{2^n}+w,\quad n>N.\]
From \eqref{p.Px.11}, it follows that
\begin{equation}\label{p.Px.fA}
 A(x)  \leqslant f(x)+(\lambda+2)w,\quad f(x)  \leqslant A(x)+(\lambda+2)w,\quad x\in\p_1.
 \end{equation}
This demonstrates that $A(x)\in (\delta v)(f(x))(\delta v)$ for all $x\in\p_1$, where $\delta=4(\lambda+2)$. We now establish the additivity of $A$.
Let $x,y\in\p_1$. By substituting $2^nx$ and $2^ny$ for $x$ and $y$ \eqref{p.Px.3} and \eqref{p.Px.4}, respectively, and then dividing both sides of the resulting inequalities by $2^n$, we obtain
\begin{align}
  \frac{f(0)}{2^n}+\frac{f\left(2^n(x+y)\right)}{2^n}& \leqslant \frac{4v}{2^n}+\frac{f(2^nx)}{2^n}+\frac{f(2^ny)}{2^n},\label{p.Px.12} \\
\frac{f(2^nx)}{2^n}+\frac{f(2^ny)}{2^n}
   & \leqslant \frac{4v}{2^n}+\frac{f\left(2^n(x+y)\right)}{2^n}+\frac{f(0)}{2^n} \label{p.Px.13}
\end{align}
for all $x\in\p_1$ and all $n\in\mathbb N$. Let $u\in\V$ be an arbitrary element. There exists a natural number $n\in\Bbb{N}$ such that
 \[\frac{f(2^nz)}{2^n}\leqslant A(z)+u,\quad A(z)\leqslant \frac{f(2^nz)}{2^n}+u,\quad z\in\{x,y,x+y\}\]
 and
 \[ \frac{f(0)}{2^n}\leqslant u,\quad  0\leqslant \frac{f(0)}{2^n}+u, \quad \frac{4v}{2^n}\leqslant u,\quad  0\leqslant \frac{4v}{2^n}+u.\]
Thus, from \eqref{p.Px.12} and \eqref{p.Px.13}, we derive
 \[A(x+y)\leqslant A(x)+A(y)+5u,\quad A(x)+A(y)\leqslant A(x+y)+5u.\]
 By applying Lemma \ref{Lem.ab}, it follows that $A(x+y)=A(x)+A(y)$. This confirms that $A$ is an additive mapping.
 \par
 Combining \eqref{p.Px.1} and \eqref{p.Px.fA}, we arrive at the following inequalities:
 \begin{align*}
  g(x)+h(0)&\leqslant v+f(x)\leqslant A(x)+(1+\delta)v, \\
   A(x)&\leqslant f(x)+\delta v\leqslant g(x)+h(0)+(1+\delta)v
 \end{align*}
 for all $x\in\p_1$. This establishes that $A(x)\in (1+\delta)v(g(x)+h(0))(1+\delta)v$ for all $x\in\p_1$.
 \par
 Similarly, by combining \eqref{p.Px.2} and \eqref{p.Px.fA}, we obtain:
 \begin{align*}
  g(0)+h(y)&\leqslant v+f(y)\leqslant A(y)+(1+\delta)v, \\
   A(y)&\leqslant f(y)+\delta v\leqslant g(0)+h(y)+(1+\delta)v
 \end{align*}
 for all $y\in\p_1$. This demonstrates that $A(y)\in (1+\delta)v(h(y)+g(0))(1+\delta)v$ for all $y\in\p_1$.
 \par
 To prove the uniqueness of $A$, let $\widetilde{A}:\p_1\to \p_2$ be another mapping  such that $\widetilde{A}(x)\in (\delta v)(f(x))(\delta v)$  for all $x\in\p_1$.  Thus,
 \[\widetilde{A}(x)\leqslant \delta v+f(x),\quad f(x)\leqslant \widetilde{A}(x)+ \delta v,\quad x\in\p_1.\]
Consequently, since $\widetilde{A}$ is additive, we obtain
\begin{equation}\label{p.Px.14}
 \widetilde{A}(x)\leqslant \frac{\delta}{2^n} v+\frac{f(2^nx)}{2^n},\quad \frac{f(2^nx)}{2^n}\leqslant \widetilde{A}(x)+ \frac{\delta}{2^n}  v,\quad x\in\p_1,~n\in\Bbb{N}.
\end{equation}
Let $u\in\V$ be an arbitrary element. There exists a natural number $n\in\Bbb{N}$ such that
 \[\frac{\delta}{2^n}v\leqslant u,\quad  0\leqslant \frac{\delta}{2^n}v+u,\quad \frac{f(2^nx)}{2^n}\leqslant A(x)+u,\quad A(x)\leqslant \frac{f(2^nx)}{2^n}+u,\quad x\in\p_1.\]
Thus from \eqref{p.Px.14},  we derive
\[\widetilde{A}(x)\leqslant A(x)+2u,\quad A(x)\leqslant\widetilde{A}(x) +2u,\quad x\in\p_1. \]
Using Lemma \ref{Lem.ab}, we conclude that $\widetilde{A}(x)=A(x)$ for all $x\in\p_1$. This proves the uniqueness of $A$.
\end{proof}

By employing Theorem \ref{Px}, we can derive the following corollaries.
\begin{cor}{\rm(Modified version of \cite[Theorem 1]{NR2})}
Consider $ (\p _1, \V_1) $ as a locally convex cone and $ (\p _2, \V_2) $ as a separated, full locally convex cone and complete with respect to the symmetric topology. Assume that  the mapping $ f: \p _1 \to \p _2 $ satisfies  the following condition
$$
2f\left(\frac{x+y}{2}\right) \in v(f(x) + f(y))v
$$
for some bounded element $ v \in \V_2 $ and for all $x, y \in \p _1$, and   $ f(0) $ is bounded. Then there exists a unique additive mapping $A:\p _1 \to \p _2$ and a positive real number $ \rho $ such that
\begin{align*}
 A(x)&\in (\rho v)(f(x))(\rho v)\;\; \text{for all}\;\;x \in \p _1.
\end{align*}
\end{cor}
\begin{proof} It is sufficient to take $g(x)=h(x):=\frac{1}{2}f(2x)$ for all $ x \in \p _1 $ in Theorem \ref{Px}.
\end{proof}
\begin{cor}{\rm(Modified version of \cite[Theorem 1]{NR1})}
Consider $ (\p _1, \V_1) $ as a locally convex cone and $ (\p _2, \V_2) $ as a separated, full locally convex cone and complete with respect to the symmetric topology. Assume that  the mapping $ f : \p _1 \to \p _2 $ satisfies the condition
$$
f(\alpha x+y) \in v(\alpha f(x) + f(y))v,\;\;x, y \in \p _1
$$
for some bounded element $ v \in \V_2 $ and for all $\alpha>0$. Then there exists a unique linear operator $L:\p _1 \to \p _2$ and a positive real number $ \gamma $ such that
\begin{align*}
 L(x)&\in (\gamma v)(f(x))(\gamma v)\;\; \text{for all}\;\;x \in \p _1.
\end{align*}
\end{cor}
\begin{proof} It is sufficient to take $g( x):=\alpha f\left(\frac{x}{\alpha}\right)$ and $h(x):=f(x)$ for all $ x \in \p _1 $ and $\alpha>0$ in Theorem \ref{Px}.
\end{proof}
\begin{cor} Let $ (\p _1, \V_1) $ be a locally convex cone. If  the mappings $ f,g,h : \p _1 \to  (\Bbb{\overline{R}}, \xi)$ $($or $f,g,h : \p _1 \to (\Bbb{\overline{R}}_+, \xi))$   satisfy the following condition
$$f(x+y) \in \varepsilon(g(x) + h(y))\varepsilon
$$for some $\varepsilon>0$ and for all $x, y \in \p _1$, and   $ f(0)<\infty$, then there exists a unique additive mapping $A : \p _1 \to  (\Bbb{\overline{R}}, \xi)$ $($or $A: \p _1 \to (\Bbb{\overline{R}}_+, \xi))$  and a positive real number $\eta $ such that
\begin{align*}
 A(x)&\in \eta(f(x))\eta, \\
 A(x)&\in (\varepsilon+\eta)(g(x)+h(0))(\varepsilon+\eta),\\
 A(x)&\in (\varepsilon+\eta)(h(x)+g(0))(\varepsilon+\eta)
\end{align*}
for all $ x \in \p _1 $.
\end{cor}
\begin{proof}
Consider the locally convex cone $(\Bbb{\overline{R}}, \xi)$ where $\xi= \{ \varepsilon> 0 : \varepsilon\in\mathbb R\}.$
This locally convex cone is a full separated locally convex cone. It is complete under
the symmetric topology.  So, the
desired result now follows from Theorem \ref{Px}.
\end{proof}
As a straightforward deduction from Theorem \ref{Px}, we obtain the following.
\begin{cor}
Consider a locally convex cone  $ (\p _1, \V_1) $ and a separated full uc-cone $ (\p _2, \V_2) $  equipped with a generating element $v$.  Assume $ (\p _2, \V_2) $ is complete under its symmetric topology.
If  the mappings $ f,g,h : \p _1 \to \p_2$   satisfy the following condition
$$f(x+y) \in (\lambda v)(g(x) + h(y))(\lambda v),\;\; \;\;x, y \in \p _1
$$
for some $\lambda>0$,  where  $ f(0)$ is bounded, then there exists a unique additive mapping $A : \p _1 \to \p_2$  and a constant $\mu >0$ such that
\begin{align*}
 A(x)&\in (\mu v)(f(x))(\mu v), \\
 A(x)&\in ((\lambda+\mu)v)(g(x)+h(0))((\lambda+\mu)v),\\
 A(x)&\in ((\lambda+\mu)v)(h(x)+g(0))((\lambda+\mu)v)
\end{align*}
for all $ x \in \p _1 $.
\end{cor}
Finally, we establish the Hyers-Ulam stability for the Pexiderized Cauchy functional equation in locally convex cones using traditional methods.
\begin{thm}
Consider a locally convex cone $ (\p _1, \V_1) $  and let $ (\p _2, \V_2) $ be a separated, full uc-cone that simultaneously forms a real vector space. Suppose $\p_2$ is complete with respect to its symmetric topology. If for some $\varepsilon  > 0$,  the mappings $ f,g,h : \p _1 \to \p_2$  satisfy
\begin{equation}\label{fgh}
f(x+y) \in (\varepsilon v)(g(x) + h(y))(\varepsilon v)
\end{equation}
for all $x, y \in \p _1$,  where $v$  is the generating element of $\V_2 $,   then there exists a unique additive mapping $A : \p _1 \to \p_2$  such that
\begin{align*}
  A(x)&\in ((4r\varepsilon+\beta) v)(f(x))((4r\varepsilon+\beta) v), \\
  A(x)&\in ((5r\varepsilon+\gamma) v)(g(x))((4r\varepsilon+\gamma) v),\\
  A(x)&\in ((5r\varepsilon+\delta) v)(h(x))((5r\varepsilon+\delta) v)
\end{align*}
for all $ x \in \p _1 ,~r>1$,  and some positive constants $\beta, \gamma, \delta$.
\end{thm}
\begin{proof}
Since the symmetric topology on $\p_2$ is Hausdorff (cf.
\cite{KR}, I.3.9.), the pair $(\p_2, \V_2)$ forms a Banach space when equipped with the norm
$q$ given in \eqref{norm}.  Additionally, because $\p_2$ is a vector space, every element of $\p_2$ is bounded, ensuring that  $q(a)<\infty$ for all $a\in\p_2$. Consequently, from \eqref{fgh}, we derive
$$f(x+y)-g(x)-h(y)\in \varepsilon (v(0)v)$$
for all $ x,y\in \p _1 $. By the definition \eqref{norm}, we have
\begin{align}\label{p1}
q(f(x+y)-g(x)-h(y))\leqslant \varepsilon
\end{align}
for all $ x,y\in \p _1 $. By using Hyers's method \cite{Hyers1941,jun, najr} and letting $y=x$ in \eqref{p1}, we obtain the inequality
\begin{align}\label{p2}
q(f(2x)-g(x)-h(x))\leqslant \varepsilon
\end{align}
for all $ x \in \p _1 $. Next, substituting $y = 0$ in \eqref{p1} yields
\begin{align}\label{p3}
q(f(x)-g(x)-h(0))\leqslant \varepsilon
\end{align}
for all $ x \in \p _1 $. Similarly, taking $x = 0$ in  \eqref{p1}  gives
\begin{align}\label{p4}
q(f(y)-g(0)-h(y))\leqslant \varepsilon
\end{align}
for all $y\in \p _1 $. Finally, setting $y=0$ in \eqref{p4} in \eqref{p4} leads to
\begin{align}\label{p5}
q(f(0)-g(0)-h(0))\leqslant \varepsilon.
\end{align}
From \eqref{p2}, \eqref{p3}, \eqref{p4} and \eqref{p5}, we obtain
\begin{align*}
q(f(2x)-2f(x)+f(0))&\leqslant q(f(2x)-g(x)-h(x))+q(-f(x)+g(x)+h(0))\\
 &\quad+q(-f(x)+h(x)+g(0))+q(f(0)-g(0)-h(0))
\\
&\leqslant 4\varepsilon
\end{align*}
for all $ x \in \p _1 $. So,
\begin{align}\label{p}
q\left(\frac{1}{2^{n+1}}f(2^{n+1}x)-\frac{1}{2^{m}}f(2^{m}x)+\sum_{k=m}^{n}\frac{1}{2^{k+1}}f(0)\right)\leqslant \sum_{k=m}^{n}\frac{2\varepsilon}{2^{k}}
\end{align}
for all $ x \in \p _1 $ and $n\geqslant m\geqslant 0$.
From \eqref{p}, it follows that the sequence $\left\{\frac{1}{2^{n}}f(2^{n}x)\right\}_{n\geqslant 1}$  forms a Cauchy sequence in the Banach space
$(\p_2, q)$.  Due to the completeness of the space, this sequence converges with respect to the norm  $q$.  We can therefore define a mapping
 $A : \p_1\to \p_2$
by taking the limit:
\[A(x)=\lim_{n\to\infty}\frac{1}{2^{n}}f(2^{n}x),\quad x \in \p _1.\]
Taking  $m=0$ and passing to the limit as $n\to\infty$ in \eqref{p}, we obtain
\begin{align}\label{p9}
q\left(A(x)-f(x)+f(0)\right)\leqslant 4\varepsilon,\quad x \in \p _1 .
\end{align}
Consequently, for all $x \in \p _1$ and $r>1$, we have   $A(x)+f(0)-f(x)\in 4r\varepsilon (v(0)v)$. This implies that
\[A(x)\in (4r\varepsilon v)(f(x)-f(0))(4r\varepsilon v),\quad x \in \p _1 ,~ r>1.\]
From   \eqref{p3} and \eqref{p9}, we obtain
\begin{align*}
q\left(A(x)+f(0)-g(x)-h(0)\right)\leqslant& q\left(A(x)+f(0)-f(x)\right)+q\left(f(x)-g(x)-h(0)\right)\leqslant 5\varepsilon
\end{align*}
for all $ x \in \p _1 $. Also, from \eqref{p4} and \eqref{p9}, we get
\begin{align*}
q\left(A(x)+f(0)-h(x)-g(0)\right)\leqslant& q\left(A(x)+f(0)-f(x)\right)+q\left(f(x)-h(x)-g(0)\right)\leqslant 5\varepsilon
\end{align*}
for all $ x \in \p _1 $.
The same way we did before, we conclude that
\[A(x)\in (5r\varepsilon v)(g(x)-f(0)+h(0))(5r\varepsilon v)\]
and
\[A(x)\in (5r\varepsilon v)(h(x)-f(0)+g(0))(5r\varepsilon v)\]
for all $ x \in \p _1 $ and all $r>1$.
Since $f(0)$, $g(0)$ and $h(0)$ are bounded,  there exist positive real numbers $\beta, \gamma$ and $\delta$ such that we  deduce
\begin{align*}
  A(x)&\in ((4r\varepsilon+\beta) v)(f(x))((4r\varepsilon+\beta) v), \\
  A(x)&\in ((5r\varepsilon+\gamma) v)(g(x))((4r\varepsilon+\gamma) v),\\
  A(x)&\in ((5r\varepsilon+\delta) v)(h(x))((5r\varepsilon+\delta) v)
\end{align*}
for all $ x \in \p _1 $ and all $r>1$. This completes the proof.
\end{proof}
\section*{Declarations}
\noindent \textbf{Conflict of Interest and Funding Disclosure }
The authors declare no conflicts of interest regarding this publication. All authors have reviewed and approved the authorship order as presented in the manuscript. This research received no substantial financial support that could have influenced its results or conclusions.

\end{document}